\documentclass[11pt]{article} 
\usepackage[utf8]{inputenc} 

\usepackage{amsthm}
\usepackage{amsmath}
\usepackage{amsfonts}
\usepackage{mathtools}
\usepackage[disable]{todonotes}
\usepackage{units}

\usepackage{float}
\usepackage{hyperref}
\hypersetup{
	colorlinks=false,
	linkbordercolor=gray,
	citebordercolor=gray,
	urlbordercolor=gray}%
\usepackage{cleveref}

\usepackage{sectsty}
\allsectionsfont{\sffamily\mdseries\scshape\centering\nohang}

\usepackage{abstract}

\usepackage{caption}
\usepackage{sidecap}

\allowdisplaybreaks

\usepackage[switch,modulo]{lineno}
\usepackage{amssymb}
\usepackage[affil-it]{authblk}

\newtheoremstyle{mplain}
{\topsep}   
{\topsep}   
{\itshape}  
{0pt}       
{\scshape} 
{:}         
{5pt plus 1pt minus 1pt} 
{}          

\renewenvironment{proof}[1][\proofname]{{ \scshape #1. }}{\qed}

\theoremstyle{mplain}

\usepackage{letltxmacro}
\makeatletter
\let\oldr@@t\r@@t
\def\r@@t#1#2{%
\setbox0=\hbox{$\oldr@@t#1{#2\,}$}\dimen0=\ht0
\advance\dimen0-0.2\ht0
\setbox2=\hbox{\vrule height\ht0 depth -\dimen0}%
{\box0\lower0.4pt\box2}}
\LetLtxMacro{\oldsqrt}{\sqrt}
\renewcommand*{\sqrt}[2][\ ]{\oldsqrt[#1]{#2}}
\makeatother

\newcommand{\E}{\mathbb{E}}

\renewcommand{\Pr}{\mathbb{P}}
\newcommand{\eps}{\varepsilon}

\newcommand{\Op}{O_\mathbb{P}}
\newcommand{\op}{o_\mathbb{P}}
\newcommand{\Omp}{\Omega_\mathbb{P}}
\newcommand{\omp}{\omega_\mathbb{P}}
\newcommand{\Tp}{\Theta_\mathbb{P}}

\newcommand{\R}{\mathbb{R}}
\newcommand{\F}{\mathcal{F}}
\newcommand{\FC}{\langle\mathcal{F}\rangle}
\newcommand{\G}{\mathcal{G}}
\newcommand{\Br}{\mathcal{B}_r}
\newcommand{\N}{\mathbb{N}}
\newcommand{\lme}{\ell}
\newcommand{\smallcost}{\lambda}
\newcommand{\bigcost}{L}
\newcommand{\holes}{\rho}

\DeclareMathOperator{\Var}{Var}
\DeclareMathOperator{\patch}{Patch}

\theoremstyle{mplain}

\newtheorem{theorem}{Theorem}[]
\crefname{theorem}{theorem}{theorems} 
\Crefname{theorem}{Theorem}{Theorems}

\newtheorem{lemma}[theorem]{Lemma}
\crefname{lemma}{lemma}{lemmas} 
\Crefname{lemma}{Lemma}{Lemmas}

\newtheorem{proposition}[theorem]{Proposition}
\crefname{proposition}{proposition}{propositions} 
\Crefname{proposition}{Proposition}{Propositions}

\crefname{cor}{corollary}{corollaries} 
\Crefname{cor}{Corollary}{Corollaries}

\newtheorem{definition}[theorem]{Definition}
\crefname{definition}{definition}{definitions} 
\Crefname{definition}{Definition}{Definitions}

\crefname{example}{example}{examples} 
\Crefname{example}{Example}{Examples}

\newtheorem{assumption}[theorem]{Assumption}
\crefname{assumption}{assumption}{assumptions} 
\Crefname{assumption}{Assumption}{Assumptions}

\newtheorem{remark}[theorem]{Remark}
\crefname{remark}{remark}{remarks} 
\Crefname{remark}{Remark}{Remarks}

\crefname{conj}{conjecture}{conjectures} 
\Crefname{conj}{Conjecture}{Conjectures}

\crefname{problem}{problem}{problems} 
\Crefname{problem}{Problem}{Problems}

\newtheorem{claim}{Claim}[]
\crefname{claim}{claim}{claims} 
\Crefname{claim}{Claim}{Claims}

\numberwithin{theorem}{section}

\title{\scshape A concentration inequality for random combinatorial optimisation problems}
	\author{Joel Larsson Danielsson}
 \affil{{ \small Department of Mathematical Sciences, Chalmers University of Technology  \\ \texttt{joel.danielsson@chalmers.se}}}

\begin{document}
\maketitle

\begin{abstract}
Given a finite set $S$, i.i.d. random weights $\{X_i\}_{i\in S}$, and a family of subsets $\F\subseteq 2^S$, we consider the minimum weight of an $F\in \F$:
\[
M(\F):= \min_{F\in \F} \sum_{i\in F}X_i.
\]
In particular, we investigate under what conditions this random variable is sharply concentrated around its mean.

We define the \emph{patchability} of a family $\F$: essentially, how expensive is it to finish an almost-complete $F$ (that is, $F$ is close to $\F$ in Hamming distance) if the edge weights are re-randomized?
Combining the patchability of $\F$, applying the Talagrand inequality to a dual problem, and a sprinkling-type argument, we prove a concentration inequality for the random variable $M(\F)$.
\end{abstract}
\section{Introduction}
\subsection{Combinatorial minimum weight problems}
The class of optimization problems that we are interested in does not necessarily involve graphs, but before giving the general definition we will first discuss them in a graph setting.
Suppose we have a finite graph $K$ (typically $K_n$ or $K_{n,n}$), a family $\G$ of subgraphs of $K$ , and a collection of i.i.d.\ non-negative random edge weights $\{X_e\}_{e\in E(K)}$. 
We are interested in the random variable $M(\G):= \min_{G\in \G} \sum_{e\in E(K)}X_e$, i.e.\ the lowest weight of a $G \in \G$.

Two famous examples are the random assignment problem and the random minimum spanning tree problem. Let $\mathcal{M}$ be the set of perfect matchings on the complete bipartite graph $K_{n,n}$, and $\mathcal{T}$ the set of spanning trees on the complete graph $K_n$. When both graphs are equipped with i.i.d. $U(0,1)$ edge weights, it has been shown that $M(\mathcal{M})\to \zeta(2)=\pi^2/6$ in probability as $n\to\infty$~\cite{matching-zeta-2}, and similarly $M(\mathcal{T})\to\zeta(3)$~\cite{tree-zeta-3}.

The spanning tree problem will be a recurring example throughout this paper, and we will prove a slight generalization of~\cite{tree-zeta-3} as an application of our concentration inequality.

A proof of convergence in probability of $M(\G)$ typically consists of two parts: First, the convergence of the expected value $\E[M(\G)]$, and then sharply concentration of $M(\G)$ around its expected value. To answer the first question one often needs a method tailored\footnote{In \cite{tree-zeta-3} a greedy algorithm was analysed to show convergence of the expected value for the minimum spanning tree problem, while in \cite{matching-larsson,wastlundmatching} a local graph limit approach was used for the random assignment problem.} to the specific family $\G$. 

In this paper, we are concerned only with the second question: When is $M(\G)$ sharply concentrated? That is, under what conditions is it true that the random variable $M(\G)$ is close to its expected value (or median) with high probability?
Our aim in this paper is to provide a `user-friendly' concentration inequality for $M(\G)$, with conditions that are easy to check.

Although we mainly have graph applications in mind, we will work in a slightly more general setting. Instead of a graph $K$ and a family $\G$ of subgraphs of $K$, we will work with a finite ground set $S$ and a family $\F\subseteq 2^S$ of subsets of $S$. We will assume $|S|=N$, and will frequently identify $S$ with $[N]=\{1,2,\ldots ,N\}$. To the elements $i\in S$, we associate i.i.d.\ random weights $X_i$, and for each set $G\subseteq S$ we let $X_G$ denote the total weight of the elements in $G$. Then, analogously to $M(\G)$ in the graph setting, we define
\[
M(\F):= \min_{F\in \F} X_F.
\]
Without loss of generality, we may assume that $\F$ contains only minimal sets: that is, if $G\subset F\in \F$, then $G\notin \F$. We also define $\lme(\F):=\max_{F\in \F}|F|$.

\subsection{Concentration inequalities}
In probabilistic combinatorics, one often needs to show that the distribution of some random variable $Z$ concentrates around some value $c$: that for any small $\eps>0$, 
$Z$ lies in the interval $[(1-\eps)c,(1+\eps)c]$ with high probability.\footnote{If such concentration holds, the median has to be close to $c$, and it is usually easy to show that the expected value is also close to $c$.}
In particular, it is a common situation that one has a product space $\Omega=\prod_{i\in S}\Omega_i$, random variables $X_i$ on the $\Omega_i$'s, and a function $g:\Omega\to \R$. One then wants to show concentration of the random variable $Z=g(X_1,X_2,\ldots X_N)$. With some abuse of notation, we can also refer to $Z(\omega)=g(X_1(\omega),X_2(\omega),\ldots X_N(\omega))$ by $g(\omega)$.

For many such functions $g$ of interest, it turns out that while $g$ does depend on all $N$ coordinates of its input, it only depends sensitively on a smaller number of coordinates.
Many concentration inequalities quantify in different ways this intuitive sense of $g$ not depending `too much' on any specific random variable $X_i$, or on any small set of these random variables.

The method of bounded differences considers the Lipschitz constant of $g$: How much can $g$ change, if only one of its inputs is changed? More precisely, we say that $g$ is $K$-Lipschitz if $|g(\omega)-g(\omega')|\leq K$ whenever $\omega,\omega'$ differ in only one coordinate. The McDiarmid inequality (based on the Azuma-Hoeffding martingale concentration inequality) bounds the size of the fluctuations around the mean by $O(K\sqrt{N})$.
This Lipschitz condition considers the worst-case change, which might very different from the \emph{typical} change: $|g(\omega)-g(\omega')|$ could be significantly smaller than $K$ for most pairs $\omega,\omega'$. (This tends to be the case for the random variable $M(\F)$, for families $\F$ that are not very small.) In a 2016 paper by Warnke~\cite{warnke-boundeddiff}, several variations on the McDiarmid inequality can be found, involving various \emph{typical-case} Lipschitz conditions. While these can greatly improve upon the inequalities based on worst-case Lipschitz constants for some functions $g$, computing the average-case Lipschitz constants for minimum-weight type problems is often not tractable.

\todo{influence inequality, discrete fourier transform, \cite{discrete-fourier-book,osss-ineq}}

Another powerful tool is the Talagrand inequality~\cite{talagrand1995concentration}, in particular the `certifiability' corollary, as found in~\cite[Thm 7.7.1]{probmeth}.
This inequality captures the intuition of a function not depending `too much' on any coordinates in a different way, by considering the `certifiability' of $g$: what is the smallest number of random variables $X_i$ that one can show to an observer to verify that the event $\{g({X})\leq s\}$ has occured (for some $s$)? In the case of minimum weight problems, $M(\F)$ is $\lme(\F)$-certifible: If $\{M(\F)\leq s\}$, then by definition there exists an $F\in \F$ with $X_F=M(\F)\leq s$, and it suffices to look at the at most $\lme(\F)$ weights of $F$ to verify that $X_F\leq s$.
A major benefit of the Talagrand inequality is that it does not depend on the dimension $N$. For $M(\F)$, it improves on the McDiarmid bound of $O(K\sqrt{N})$, down to order $O(K\sqrt{\lme(\F)})$. As far as we are aware, Talagrand-type inequalities have only been established for worst-case Lipschitz constants.

Let's consider a naive application of the bounded difference method and the Talagrand inequality to the minimum spanning tree problem. Here $N=\binom{n}{2}$, $\lme(\mathcal{T})=n-1$ and $K=1$, so the bounded difference method gives that $M(\mathcal{T})$ has a standard deviation of $O(n)$, which the Talagrand inequality lowers to $O(\sqrt{n})$. However, $\E[M(\mathcal{T})]=O(1)$ as $n \to \infty$ (since it converges to $\zeta(3)$), so neither bound is useful.
Both of these inequalities suffer from using the worst case Lipschitz constant $K=1$, while the typical change in $M(\mathcal{T})$ when changing one edge weight is of order\footnote{Follows from the proof of theorem \ref{thm:mst} with $r=1$.} $1/n$.

It is easy to apply our patchability inequality to the minimum spanning tree problem. This gives an upper bound of $O(n^{-1/4})$, implying sharp concentration (see section \ref{section:mst}). The `patchability' criterion (definition \ref{def:patch}) behaves more like the average case Lipschitz constant than the worst case.
However, this only meant as a comparison between these three concentration inequalities. Much stronger results have been obtained previously, for instance a central limit theorem for $M(\mathcal{T})$ was established in~\cite{janson-mst-clt}, with a standard deviation of order $n^{-1/2}$.
\footnote{For instance, in~\cite{frieze-spanningtree-cost-and-weight} a log-Sobolev inequality is used}

\subsection{Asymptotic notation}
In addition to the commonly used asymptotic notation of $O,o,\omega,\Omega$, we will also use $\Op,\op,\omp$ and $\Omp$ to denote the probabilistic versions:
For sequences $X_n,Y_n$ of random variables, we say that $X_n=\op(Y_n)$ and $Y_n=\omp(X_n)$ if $X_n/Y_n\to 0$ in probability as $n\to \infty$. We say that $X_n=\Op(Y_n)$ and $Y_n=\Omp(X_n)$ if there for any $\eps>0$ exists a constant $C>0$ such that $X_n\leq CY_n$ with probability at least $1-\eps$ for all sufficiently large $n$. Furthermore, $X_n=\Tp(Y_n)$ iff $X_n=\Op(Y_n)$ and $X_n=\Omp(Y_n)$.

Finally, we use $f(n)\ll g(n)$ to denote $f(n)=o(g(n))$. And unless otherwise specified, the asymptotics will always be implicitly `as $n\to \infty$' (or `as $N\to \infty$').

\section{Results}
\subsection{Patchability condition}
\label{section:patchability}
Loosely speaking, our concentration inequality says that if any `almost-complete' member of $\F$ (missing on the order of $\sqrt{\lme(\F)}$ elements) can be completed at cost $o(\E M(\F))$ (whp), then the optimal cost $M(\F)$ has to be sharply concentrated.
Before stating the theorem, we need to make this notion more precise.

As noted earlier, we can assume without loss of generality that $\F$ contains only minimal sets.
We will let $\FC$ denote the upwards closure of $\F$: \(\FC=\{G\subseteq S: \exists F\in \F: F\subseteq G\}\). 

Since the weights are non-negative, $M(\F)=M(\FC)$.

For any $G,P\subseteq S$, we say that $P$ is a \emph{$G$-patch} if $G\cup P $ contains a member of $\F$, i.e. $G\cup P \in \FC$. Define the function $\holes:2^S\mapsto \N$ by 
\begin{align}
\holes(G):=d(G,\FC) = \min\{|P|: P\textrm{ is a $G$-patch}\}\label{def:holes},
\end{align}
where $d$ denotes Hamming distance: $d(G,F)=|G\Delta F|$ and $d(G,\F)=\min_{F\in F}|G\Delta F|$. 

Let $\Br=\Br(\FC)$ be the $r$-neighbourhood of $\FC$ in the Hamming metric, i.e.\ the set of all $G\subseteq S$ with $\holes(G)\leq r$.
For any set $G\subseteq S$, let the random variable $\mathrm{Patch}(G)$ be the minimum weight of a $G$-patch:
\[\mathrm{Patch}(G):= \min\{X_P: P \textrm{ is a }G\textrm{-patch}\}.\]

\begin{definition}
We say that a $G\subseteq S$ is $(\smallcost,\eps)$-\emph{patchable} (with respect to the random weights $X_i$) if $G$ can be patched at cost at most $\smallcost$ with probability at least $1-\eps$. That is,
\label{def:patch}
\[\Pr(\mathrm{Patch}(G)\leq \smallcost)\geq 1-\eps.\]
The family $\F$ is said to be $(r,\smallcost,\eps)$-\emph{patchable} if every $G\in \Br$ (that is, $G\subseteq S$ with $\holes(G)\leq r$) is  $(\smallcost,\eps)$-patchable.
\end{definition}

\begin{remark}
$\mathrm{Patch}(G)$ can also be seen as a Hamming distance to $\FC$: if we define the randomly weighted Hamming distance by $D(G,F):=X_{G\Delta F}$, then $\mathrm{Patch}(G)=D(G,\FC)$.
In terms of these Hamming distances, $\F$ is $(r,\smallcost,\eps)$-patchable if any $G$ in the $r$-neighbourhood of $\FC$ w.r.t.\ the metric $d$ lies in the $\smallcost$-neighbourhood of $\FC$ w.r.t.\ the metric $D$, with probability at least $1-\eps$.
\end{remark}
In other words, $\F$ is $(r,\smallcost,\eps)$-\emph{patchable} if any $F\in \F$ from which an arbitrary $r$ elements has been removed (giving us a $G$ with $\holes(G)=r$), can be patched at cost at most $\smallcost$, with probability at least $1-\eps$. Patching $G$ will not necessarily restore the \emph{same} $F$, as we only require that our `patched' set $G\cup P$ contains \emph{some} member of $\F$.
It is important to note here that the set $G$ is not random, and in particular it is not chosen in a way that depends on the weights $X_i$.

When applying our inequality (theorem \ref{thm:concentration}), the main effort will usually be to show that this patchability condition is met, for some $r,\smallcost$ and $\eps$.
Note that if $X_i\leq 1$ with probability $1$, then any family $\F$ is trivially $(r,r,0)$-patchable for any $r$ -- simply put back the $r$ elements that were removed, at cost at most $r$. However, it is only for families that are $(r,\smallcost,\eps)$-patchable for some $\smallcost \ll r$ that our inequality (theorem \ref{thm:concentration}) improves upon a `standard' application of the Talagrand inequality.

\subsection{Probability distributions of the weights}
Most commonly, the edge weights $X_i$ are exponential or uniform $U(0,1)$. But the proof of our inequality can easily be generalised to a larger class of edge weight distributions, such as the positive powers of such random variables. The only assumption we will make on the distribution of $X_i$ is the following.

\begin{assumption}
\label{assumption:Aq}
For $q>0$ we say that the distribution of a random variable $X$ satisfies assumption $A(q)$ if the following holds:
For any $s\in (0,1)$, $X$ can be coupled to two copies $Y\overset{d}{=}Y'\overset{d}{=}X$ following the same distribution, such that $Y$ and $Y'$ are independent and surely
\begin{equation}
\label{eq:Aq}
     X \leq \min\left( \frac{Y}{(1-s)^{1/q}},\frac{Y'}{s^{1/q}}\right).
\end{equation}
\end{assumption}
For instance, the $(1/q)$:th power of exponential or uniform random variables satisfies $A(q)$. The assumption is also closely related to the so-called \emph{pseudo-dimension} of a random variable. We will discuss this in section \ref{section:mindominant}.
\begin{remark}
\label{remark:Aq-iter}
Note also 
that \cref{eq:Aq} can be iterated, so that for any $\{s_i\}_{i=1}^k$ with $\sum_{i=1}^k s_i=1$ we have that i.i.d. random variables $Y^{(i)}\overset{d}{=}X$ can be coupled to $X$ in such a way that $X \leq \min_{i\in [k]}\left({Y^{(i)}}/{s_i^{1/q}}\right).$ 
In particular, for $s_i=k^{-1}$, $X \leq k^{1/q}\min_{i\in [k]}\left({Y^{(i)}}\right).$
\end{remark}

\subsection{A `patchability' concentration inequality}
We are now ready to state our concentration inequality.
\label{section:theinequality}
\begin{theorem}
\label{thm:concentration}
Assume the distribution of the $X_i$'s satisfies $A(q)$ in assumption (\ref{assumption:Aq}).
Given $\eps>0$, let $\bigcost$ be such that $\Pr(M(\F)>\bigcost)\geq  \eps$.  
If $\F$ is $(r,\smallcost,\eps)$-patchable for some $r\geq \sqrt{8\log(\eps^{-1})\cdot \lme(\F)}$ and $\smallcost>0$, then with probability at least $1-2\eps$,
\begin{align}
\label{ineq:UB-thm}
    M(\F)&\leq \left(\bigcost^{\frac{q}{q+1}}+\smallcost^{\frac{q}{q+1}} \right)^{\frac{q+1}{q}}.
    \end{align}
In particular, if $m$ is the median of  $M(\F)$, $\eps<1/4$, and $\smallcost\leq m$, then for some constant $C=C(q)$, 
with probability at least $1-3\eps$,
\begin{align}
 \frac{|M(\F)-m|}{m} &\leq C \left(\frac{\smallcost}{m}\right)^{\frac{q}{q+1}}.
\end{align}
\end{theorem}

The theorem gives sharp concentration if $\smallcost = o(m)$ as $N\to \infty$.

Note also that $r$ is an increasing function of $\lme(\F)$, the largest size of a member of $\F$, and $\smallcost$ in turn is a non-decreasing function of $r$. One therefore want to make $\lme(\F)$ as small as possible.

A less general version of theorem \ref{thm:concentration},  appeared in a previous paper~\cite{minHfactor} by the author and  L.~Federico. There we studied a specific family $\F$ (the family of $H$-factors in $K_n$, for some small graph $H$), for which a property very similar to patchability holds trivially. We discuss this family in section \ref{min-H-factor}.

\begin{remark}
\label{remark:bigandsmallsets}
Some families $\F$ have a large variation in the sizes of its members, and while smaller sets $F$ are more likely to have low weight $X_F$, there might be a much larger number of bigger sets in  $\F$. If the $F$ achieving the optimal weight $M(\F)$ tends to be a small set, it can sometimes be helpful to sort the members of $\F$ into two families: $\F_{small}$ with small sets and $\F_{big}$ with big sets.
If one can (with high probability) upper bound $M(\F_{small})< t$ and lower bound $M(\F_{big})> t$ with the same $t=t(n)$ (for instance by the methods discussed in sections \ref{section:ub-lb} and \ref{section:uppertail}), then with high probability $M(\F)=M(\F_{small})$ and it suffices to show concentration of $M(\F_{small})$. 
\end{remark}

\subsection{Proof strategy for theorem \ref{thm:concentration}}
This proof will follow a similar strategy to that in~\cite{minHfactor}.
The main novel ideas (here and in~\cite{minHfactor}) are (i) the patchability condition, and (ii) to apply the Talagrand inequality to a dual problem:
Setting a `budget' $\bigcost>0$, how close (in Hamming distance) to a member of $\F$ can we get while staying within budget? 
More precisely, we define
\begin{equation*}
Z_\bigcost := \min\{\rho(G): G\subseteq S \textrm{ and } X_G\leq L \}.
\end{equation*}
Talagrand's concentration inequality is much better suited to this random variable, and we use it to show that $Z_\bigcost$ is typically `small' (roughly of order $\sqrt{\lme(\F)}$) for a suitable $\bigcost$. 
That is, that there exists a $G\subseteq S$ with weight $X_G\leq \bigcost$, and which can be turned into a member of $\F$ by adding at most a small number of elements from $S$. 

For the next step, given such a $G$,
we would like to find a cheap `patch' $P\subseteq S$ such that $G\cup P\in \F$. However, here we run in to an obstacle: $G$ is now a random set, and the weights of the elements not in $G$ will not be independent from $G$, because $G$ was chosen in a way that depends on the weights $X_i$.

To get around this obstacle, we perform a trick originally due to Walkup~\cite{walkup1979expected}, which we call the \emph{red-green split}.
Split each element $x\in S$ into two, a green and a red copy. 
For some small $s>0$, give these independent random weights ${Y_i}/({1-s})^{1/q},{Y_i'}/{s}^{1/q}$, where $Y_i$ and $Y_i'$ follow the same distribution as $X_i$, and couple them to $X_i$ as in assumption $A(q)$.
With this coupling, the green weights $Y_i$ are typically close to $X_i$, while the red weights $Y'_i$ tend to be larger. Crucially, the red weights are independent from the  green weights.

We then study the dual problem on the green weights, and use Talagrand's inequality as described above to show that there probably exists a green $G\subseteq S$ with $\rho(G)\leq r$ and $Y_G\leq L$, with $r$ of order roughly $\sqrt{\lme(\F)}$. 
Next we use that $\F$ is $(r,\smallcost,\eps)$-patchable to find a red $G$-patch $P$ with $Y'_P\leq \smallcost$. Since $G\cup P\in \FC$, we have that
\begin{equation}
\label{ineq:intro-redgreen}
M(\F)\leq X_G+X_P\leq \frac{Y_G}{(1-s)^{1/q}}+\frac{Y_P'}{s^{1/q}}.
\end{equation}
Using that $Y_G\leq L$ and $Y'_P\leq \smallcost$ (with high probability), and optimizing over $s$ gives us the upper bound (\ref{ineq:UB-thm}) in theorem \ref{thm:concentration}.

\section{Applications}
\subsection{Minimum spanning tree}
\label{section:mst}
\begin{theorem}
\label{thm:mst}
Let $\mathcal{T}$ be the family of spanning trees on $K_n$. For some $q>0$, equip $K_n$ with i.i.d.\ edge weights following the distribution of the $(1/q)$:th power of a uniform $U(0,1)$ random variable.
Then there exists a constant $c_q$ such that for $\bigcost:=c_q n^{1-1/q}$, 
\[\frac{|M(\mathcal{T})-\bigcost|}{\bigcost}=\Op\Big(n^{-\frac{q}{2(q+1)}}\Big)=\op(1).\]
\end{theorem}
In particular, for $q=1$, it is known that $\E[M(\mathcal{T})]\to\zeta(3)\approx 1.202$, and  $\Var[M(\mathcal{T})]=\Theta(1/n)$, so that the fluctuations of $M(\mathcal{T})$ around its expected value is of order $n^{-1/2}$. Our theorem gives a weaker upper bound, of order $n^{-1/4}$.

For  $q\geq 1$ the theorem follows from \cite[Section 5]{frieze-spanningtree-cost-and-weight}, but it seems to be novel for $q<1$.

\begin{proof}
In \cite{tree-zeta-3} it is shown that $\E[M(\mathcal{T})]\to \zeta(3)$ when $q=1$. Assuming that the edge weights are such that $X_i^q\sim U(0,1)$, it is easy to adapt this argument to show that $\E[M(\mathcal{T})]/n^{1-1/q}$ converges to a constant as $n\to\infty$.

We will prove that $\patch(G)=\Op(rn^{-1/q})$ for any $G$ with $\holes(G)=r$. The theorem then follows by plugging this into theorem \ref{thm:concentration}, by noting that (i) since a spanning tree has $n-1$ edges, $r=\Theta(\sqrt{n})$ (for $\eps$ fixed), and (ii) sharp concentration of $M(\mathcal{T})$ around its median implies that the expected value is close to the median.

Pick any $G$ with $\holes(G)=r=\Theta(\sqrt{n})$.
A graph with $\holes(G)=r$ has $r+1$ connected components, as $r$ edges must be added to it in order to connect it. 
Let $C_1,\ldots, C_{r+1}$ be these connected components, sorted in increasing order by their number of vertices, and with ties broken arbitrarily.
For each edge in $E(K_n)-E(G)$ that goes between two components, orient it according to the order of the components above: from $C_i$ to $C_j$ when $i<j$.
We will find a $G$-patch $P$ by, for each $1\leq i\leq r$, picking the cheapest outgoing edge from $C_i$. Note that such a $P$ is indeed a $G$-patch, because $\langle\mathcal{T}\rangle$ is the family of connected graphs, and there is a path in $G\cup P$ from any $C_i$ ($i\leq r$) to $C_{r+1}$. 
\begin{claim}
\label{claim:patch-trees}
For $1\leq i\leq r$, $C_i$ has at least $s:=\min(n/2,n^2/4r^2)=\Theta(n)$ outgoing edges. 
\end{claim}
\begin{proof}
Consider a component $C_i$ with $k$ vertices. If $k\leq n/2r$, then $C_1,\ldots ,C_i$ all have at most $k$ vertices each, for a total of at most $rk\leq n/2$ vertices. Hence there are at least $n/2$ total vertices in $C_{i+1}\cup  \ldots \cup C_{r+1}$, and every vertex in $C_i$ has at least this many outgoing edges. If instead $k>n/2r$, $C_{i+1}$ also has at least $k$ vertices, so there are at least $k^2>n^2/4r^2$ edges from $C_i$ to $C_{i+1}$.
\end{proof}

Let $W_i$ be the minimum weight of an outgoing edge from $C_i$. The $W_i$'s are independent. 
The minimum of $m$ i.i.d. edge weights $X_i$ such that $X_i^q\sim U(0,1)$ has expected value and standard deviation of order $\Theta(m^{-1/q})$.
Hence the $W_i$'s have expected value and standard deviation uniformly bounded by some $O(n^{-1/q})$, and there exists a constant $c>0$ such that with high probability \[\patch(G)\leq \sum_{i=1}^r W_i \leq crn^{-1/q}.\]
Hence $\mathcal{T}$ is $(\smallcost,r,\eps)$-patchable with $\smallcost =crn^{-1/q}$ and $r=\Theta(\sqrt{n})$.
By theorem \ref{thm:concentration} we have that $|M(\mathcal{T})-m|/m\leq \Op((\smallcost/m)^{\frac{q}{q+1}})$. Here $m=\Theta(n^{1-1/q})$ and $\smallcost=O(n^{\frac{1}{2}-1/q})$, so that $\smallcost/m = O(r/n)=O(n^{-\frac{1}{2}})$. Hence $|M(\mathcal{T})-m|/m  \leq \Op(n^{-\frac{q}{2(q+1)}})=\op(1)$.
\end{proof}

\subsection{Minimum $H$-factor}
\label{min-H-factor}
Given a fixed graph $H$, an $H$-factor (or tiling) on $K_n$ is a collection of vertex-disjoint copies of $H$, which together cover all vertices of $K_n$. For $H=K_2$, this is the random assignment (also known as minimum perfect matching) problem which we discuss in the next subsection.
 In a paper by the author and L. Federico~\cite{minHfactor}, an earlier version of theorem \ref{thm:concentration} was used to show sharp concentration of the minimum weight of an $H$-factor for graphs $H$ containing at least one cycle.

In the minimum $H$-factor problem, patches have a particularly nice structure: They are essentially $H$-factors on smaller vertex sets. If $F$ is an $H$-factor and we remove $r$ edges, we may as well remove the (at most) $r$ copies of $H$ these edges belonged to. This leaves a partial $H$-factor $G$, and any $H$-factor on the at most $r\cdot v(H)$ uncovered vertices forms a $G$-patch.

For graphs $H$ containing at least one cycle (and random weights satisfying $A(1)$, such as $U(0,1)$), we showed that the minimum weight $M$ of an $H$-factor is of order $\Tp(n^{\beta})$ w.h.p., for some $\beta=\beta(H)\in (0,1)$ . This immediately implies that $H$-factors are $(r,\lambda,\eps)$-patchable where $\lambda =O(r^{\beta})$. When applying theorem \ref{thm:concentration}, $r$ is of order $\sqrt{n}$, so that $\lambda$ is of order $n^{\beta/2}$. Since the median $m$ of $M$ is of order $m=\Theta(n^\beta)$, the theorem gives us that $|M-m| \leq \Op(\sqrt{\lambda m})= \Op(m^{\frac{3}{4}})$. In other words, $M$ is sharply concentrated.

\subsection{Random assignment}
In the random assignment problem, $\F$ is the set of perfect matchings on the complete bipartite graph $K_{n,n}$.
This problem has been studied when the edge weights satisfy condition $A(q)$ in (\ref{eq:Aq}), for $q=1$~\cite{matching-zeta-2}, $q>1$~\cite{wastlundmatching} and by the present author for $q<1$~\cite{matching-larsson}. In all cases it has been shown that $M(\F)/n^{1-1/q}$ converge in probability to a constant depending only on $q$.

However, a straight-forward application of theorem \ref{thm:concentration} only gives sharp concentration in the case $q>1$. Although a perfect matching is also an $H$-factor (with $H=K_2$), the minimum weight scales like $n^{1-1/q}$, and a similar argument to that in section \ref{min-H-factor} only leads to sharp concentration if the exponent is positive.

\subsection{Minimum spanning $d$-sphere}
A combinatorial $d$-sphere is a $(d+1)$-regular hypergraph, which -- when viewed as the set of maximal faces of an abstract simplicial complex -- is homeomorphic to a $d$-sphere. In an upcoming paper joint with A. Georgakopoulos and J. Haslegrave, we study the minimum weight of a spanning $d$-sphere in a randomly-weighted complete $(d+1)$-uniform hypergraph $K_{n}^{(d+1)}$~\cite{min-d-sphere}, again with weights satisfying $A(1)$. We show concentration of this minimum weight for $d=2$ and $3$, and the proof for $d=2$ uses theorem \ref{thm:concentration}.

\section{Proofs}
\label{proofs}
\label{section:concentration}

To prove theorem \ref{thm:concentration} we will need the following lemma, which is where the `red-green split trick' is used. 
Recall that for any $r\geq 0$, $\Br$ is the set of $G\subset S$ within Hamming distance $r$ of $\FC$, i.e.\ such that there exists a $P\subseteq S$ with $|P|\leq r$ and $G\cup P\in \FC$.
\begin{lemma}
Assume $a,b>0$ and pick $c$ such that $c^{\frac{q}{q+1}}=a^{\frac{q}{q+1}}+b^{\frac{q}{q+1}}$.
 Let $G^*$ be the (random) set in $\Br$ with minimal cost, i.e.\ $W_{G^*}=M(\Br)$. Then
\label{redgreensplit}
\begin{align}
   \Pr\big(M(\F)>c\big) \,&\leq \,\Pr\big(M(\Br)>a\big)+\Pr\big(\mathrm{Patch}(G^*)>b\big) \nonumber
   \\
   &\leq  \, \Pr\big(M(\Br)>a\big)+\max_{G\in \Br}\Pr\big(\mathrm{Patch}(G)>b\big).
   \label{ineq:redgreensplit}
\end{align}
\end{lemma}
We will also need the following claim.
\begin{claim}
\label{claim:ab-min}
Let $f(s):=\frac{a}{(1-s)^p}+\frac{b}{s^p}$, with $a\geq b>0$ and $p\geq 0$. Then $f$ has a unique minimum $s_0$ on $(0,1)$, with
\[
f(s_0)= (a^{\frac{1}{p+1}}+b^{\frac{1}{p+1}})^{p+1}\leq a\cdot \big(1+C\cdot (b/a)^{\frac{1}{p+1}}\big),
\]
for some constant $C=C(p)$. In particular if $a\gg b$, $f(t)=a \cdot (1+o(1))$.
\end{claim}
We postpone the proofs of lemma \ref{redgreensplit} and claim \ref{claim:ab-min} until the end of this section.
Now, let's instead proceed with the proof of our main theorem.
\begin{proof}[Proof of theorem \ref{thm:concentration}]
We want to upper bound the probability that $M(\F)$ is large. To do this, we will use lemma \ref{redgreensplit} with $a=\bigcost,b=\smallcost$.
\begin{equation} 
\label{eq:redgreen-Ll}
\Pr\big(M(\F)>c\big)  \leq  \, \underbrace{\Pr\big(M(\Br)>\bigcost\big)}_{\textrm{Upper bound by Talagrand's inequality}}+\underbrace{\max_{G\in \Br}\Pr\big(\mathrm{Patch}(G)>\smallcost\big)}_{\textrm{Upper bound by using patchability}}
\end{equation}
For the second term of the right-hand side of (\ref{eq:redgreen-Ll}), we use the patchability condition: By assumption, $\F$ is $(r,\smallcost,\eps)$-patchable, or in other words $\mathrm{Patch}(G)>\smallcost$ with probability at most $\eps$ for all $G\in\Br$.

Recall that \(Z_\bigcost:=\min\{\holes(G): X_G\leq \bigcost\}\),
and note that $M(\Br)>\bigcost$ if and only if $Z_\bigcost>r$.

We now want to apply the Talagrand inequality to the first term on the right-hand side of (\ref{eq:redgreen-Ll}). The way this inequality is stated in~\cite{probmeth}, we would need to apply it to the random variable $-Z_\bigcost$. However, for the sake of clarity and to avoid a clutter of minus signs, we reformulate the inequality and the definition of certifiability so that they apply directly to $Z_L$. The random variable $Z_L$ has the following two properties:
\begin{description}
\item[$Z_\bigcost$ is $1$-Lipschitz:] Suppose $\omega,\omega'\in \Omega$ are such that $X_i(\omega)=X_i(\omega')$, for all $i$ except some $i_0$. Consider $G\subseteq S$ such that $X_G(\omega)\leq \bigcost$ and $\rho(G)$ attains the minimum $Z_\bigcost(\omega)$. Then $G':=G-\{i_0\}$ satisfies $X_{G'}(\omega')=X_{G'}(\omega)\leq L$, and $\rho(G') $ is at most $\rho(G)+1$, so that $Z_\bigcost(\omega')\leq Z_\bigcost(\omega)+1$. By interchanging $\omega$ and $\omega'$, $|Z_\bigcost(\omega')- Z_\bigcost(\omega)|\leq 1$.

\item[$Z_\bigcost$ is $\lme(\F)$-certifiable:] If $\omega$ is such that $Z_\bigcost(\omega)\leq s$, there exists a $G$ with $X_G(\omega) \leq \bigcost$ and $\holes(G)\leq s$. Assuming WLOG that $G$ is a minimal such set, it has at most $\lme(\F)$ elements. These are a certificate that $Z_\bigcost\leq s$: any $\omega'$ which agrees with $\omega$ on the set $G$ has $X_G(\omega')=X_G(\omega)\leq \bigcost$, and hence $Z_L(\omega')\leq s$ too.
 \end{description}
The Talagrand inequality then states that for any $t>0$ and $b$,
\begin{equation}
    \label{ZL:talagrand}
    \Pr\left(Z_\bigcost\leq b\right) \cdot  \Pr\left(Z_\bigcost\geq b+t\sqrt{\lme(\F)}\right) \leq e^{-t^2/4}.
\end{equation}
Let $t:=\sqrt{8\log (\eps^{-1})}$, so that $e^{-t^2/4}=\eps^2$ and $r\geq t\sqrt{\lme(\F)}$, and let $b:=0$.
The first probability on the left-hand side of (\ref{ZL:talagrand}) is $\Pr(Z_\bigcost=0)=\Pr(M(\F)\leq \bigcost)$, which is at least $\eps$ by assumption. Hence the second probability is $\Pr(Z_\bigcost\geq r)\leq \eps$.

The first term on the right-hand side of (\ref{eq:redgreen-Ll}) is then
$\Pr(M(\Br)>\bigcost)=\Pr(Z_\bigcost>r)\leq \eps$, and hence $\Pr(M(\F)>c)\leq 2\eps$.

The `in particular'-statement follows from the second part of claim \ref{claim:ab-min} with $p=q^{-1}$, $a=L$, and $b=\lambda$: For some constant $C$, $c=(\bigcost^{\frac{q}{q+1}}+\smallcost^{\frac{q}{q+1}})^{\frac{q+1}{q}}\leq \bigcost
\cdot (1+C\cdot (\smallcost/\bigcost)^{\frac{q}{q+1}})$. Then $M(\F)> c$ with probability at most $2\eps$, and $M(\F)< L$ with probability at most $\eps$ by assumption. Hence $M(\F)$ lies within an interval of length $LC\cdot (\lambda/L)^{\frac{q}{q+1}}$ with probability $1-3\eps$, and in particular the median also lies in this interval.
\end{proof}

\begin{proof}[Proof of lemma \ref{redgreensplit}]
For convenience, set $p=q^{-1}$.
For a small $s>0$ to be chosen later, use assumption $A(1)$ in assumption \ref{assumption:Aq} to couple the weights $X_i$ to a pair of independent random variables $Y_i,Y_i'\overset{d}{=}X_i$ such that $X_i\leq {Y_i}/({1-s})^p+{Y_i'}/{s^p}$. We think of the $Y_i$ as `green' weights and the $Y'_i$ as `red' weights.
Using the coupling of $Y_i,Y_i'$ in (\ref{eq:Aq}) gives that for any $F=G\cup P\in \FC$, surely
\[
M(\F)\leq X_F\leq X_G+X_P\leq \frac{Y_G}{(1-s)^p}+\frac{Y'_P}{s^p}.
\]
Our strategy is now to find a cheap green $G$ with $\rho(G)\leq r$, and then find a cheap red $G$-patch $P$.
Let $W$ be the cost of the cheapest such $G$ w.r.t.\ the weights $Y_i$, or in other words $W :={\min\{Y_G:G\in \Br\}}$, and let $G^*$ be the random set $G$
which achieves this minimum.
Similarly, for any $G\in \Br$, let the random variable $W'(G)$ be the minimum of $Y'_P$ over all $G$-patches $P$. Then
\begin{equation}
\label{ineq:M-Xgreen-Xred}
M(\F)\leq \frac{W}{(1-s)^p}+\frac{W'(G^*)}{s^p}.
\end{equation}
Now,  (\ref{ineq:M-Xgreen-Xred}) is at most $\frac{a}{(1-s)^p}+\frac{b}{s^p}$, 
unless $W>a$ or $W'(G^*)>b$. 
Note that since $Y_i,Y'_i$ follow the same distrubition as $X_i$, ${W}\overset{d}{=}M(\Br)$ and ${W'}(G)\overset{d}{=}\mathrm{Patch}(G)$.
Hence the right-hand side of (\ref{ineq:redgreensplit}) is $\alpha + \beta$, with
\begin{align*}
    \alpha:=&\Pr(M(\Br)>a)=\Pr(W>a),
    \\
    \beta:=&\max_{G\in \Br} \Pr(\mathrm{Patch}(G)>b)=\max_{G\in \Br} \Pr(W'(G)>b).
\end{align*}
For the second termof the right-hand side of (\ref{ineq:M-Xgreen-Xred}), by the choice of $\beta$, $\Pr\left(W'(G)>b\right)\leq \beta$ for any $G$.
By averaging, it also holds for $G$ picked according to a any probability distribution on $\Br$ which is independent from the red weights. 
In particular, it holds if $G^*$ is the $G$ that achieves the minimum $W=\min_{G\in \Br}Y_G$, since $G^*$ only depends on the green weights. Hence $\Pr\left(W'(G^*)>b\right)\leq \beta$.
So by a union bound on the right-hand side of (\ref{ineq:M-Xgreen-Xred}),
\begin{equation}
\Pr\left(M(\F)> \frac{a}{(1-s)^p}+\frac{b}{s^p}\right)\leq \alpha + \beta.
\label{ineq:M-redgreenbound}
\end{equation}
But since $s\in (0,1)$ was arbitrary, we can minimize $\frac{a}{(1-s)^p}+\frac{b}{s^p}$ over $s$. 

Noting that $(a^{\frac{1}{p+1}}+b^{\frac{1}{p+1}})^{p+1}=(a^{\frac{q}{q+1}}+b^{\frac{q}{q+1}})^{\frac{q+1}{q}}=c$, the lemma follows.
\end{proof}

\begin{proof}[Proof of claim \ref{claim:ab-min}]
Since $f$ is smooth and strictly convex on $(0,1)$, there exists a unique local minimum $s_0$, which is also the global minimum.
Let $u:=a^{\frac{1}{p+1}},v:=b^{\frac{1}{p+1}}$.
Then $f'(s) = {-p\big(\frac{u^{p+1}}{(1-s)^{p+1}}-\frac{v^{p+1}}{s^{p+1}}\big)}$, which is zero iff $\frac{u}{1-s}=\frac{v}{s}$. Solving for $s$ gives $s_0:=\frac{v}{u+v}$, so that  $f(s_0)={(u+v)^{p+1}}$.

The function $\varphi(x)=(1+x)^{p+1}$ is convex, and hence it lies below the secant line with intersections at $x=0$ and $x=1$. This secant line has slope $C:=\varphi(1)-\varphi(0)=2^{p+1}-1$. In other words, $(1+x)^{p+1}\leq 1+Cx$ for any $x\in [0,1]$. With $x:=v/u\leq 1$, we get $(u+v)^{p+1}\leq u^{p+1}(1+Cv/u)$.
\end{proof}

\section{Other weight distributions}
\label{section:mindominant}
Assumption \ref{assumption:Aq} is closely related to the so-called \emph{pseudo-dimension} of a distribution. If the cdf $F$ of $X$ is such that, for some $d>0$, $F(x)/x^d$ converges to some $c\in (0,\infty)$ as $x\to 0$, then $X$ is said to be of pseudo-dimension $d$. 
The motivation behind the name is that if $d$ is a positive integer and two points are chosen uniformly at random from the $d$-dimensional unit box, the distribution of the Euclidean distance between these points is of pseudo-dimension $d$.

Since pseudo-dimension is only a condition on the behavior of $F(x)$ near $0$, we cannot guarantee that assumption $A(q)$ will hold. It is, however, often the case that the distribution of $M(\F)$ is (asymptotically) the same for any weights of pseudo-dimension $q$, up to a global rescaling. If one has weights $X_i$ of pseudo-dimension $q$ but not satisfying $A(q)$, it is usually easiest to first show that one can approximate these with e.g. the $(1/q)$:th power of $U(0,1)$-distributed random variables, and then apply theorem \ref{thm:concentration}.

We will now briefly outline one potential strategy to do this, using a variant of the patchability condition. Start with the $F$ achieving optimality ($X_F=M(\F)$). Remove expensive elements (weight $\geq \delta$ for some $\delta>0$) from $F$, resulting in some $G$ with $\holes(G)$ fairly small. Rerandomize the edge weights, and search for a $G$-patch which is both cheap, \emph{and} uses no edge of cost above $\delta$. If (whp) such a patch can be found, then one can show that there is an $F'\in \F$ using only cheap elements, and with $X_{F'}$ very close to $M(\F)$. Since $F'$ only uses cheap elements, we can couple the weights of pseudo-dimension $q$ to, for instance, the $(1/q)$:th power of $U(0,1)$-weights, and thereby show that $M(\F)$ with pseudo-dimension $q$ weights can be well approximated by $M(\F)$ with weights satisfying $A(q)$.

For an example of a proof following this strategy, see theorem 5.1 in~\cite{minHfactor}.

\section{Bounds on $M(\F)$}
\label{section:ub-lb}
As noted in section \ref{section:theinequality}, $M(\F)$ is sharply concentrated if $\F$ is $(r,\smallcost,\eps)$-patchable with $\eps$ small, $r$ at least of order $\sqrt{\lme(\F)}$, and $\smallcost =o(m)$ as $N\to\infty$ (where $m$ is the median of $M(\F)$). To verify that $\smallcost =o(m)$, one typically needs a lower bound on $m$. In section \ref{section:lowerbound} we provide a generic first moment method bound, which in practice often turns out to be within a constant factor of the true $m$.

\subsection{Upper bound}
While not strictly necessary to prove sharp concentration, one is usually also interested in finding a matching upper bound on $m$ (or $M(\F)$).
 These tend to require an approach tailored to the specific family $\F$ one is studying. Here are some examples of approaches that have been successful in the past.
 \begin{enumerate}
    \item For all $F\in \F$, $M(\F)\leq X_F$. Any algorithm for finding an $F$ with low weight will give an upper bound on $M(\F)$, and even fairly naive algorithms (e.g.\ greedy algorithms) can often be within a constant factor of optimal. Or, in the case of minimum spanning tree, actually optimal.
    \item For any $\F\subseteq \F$, $M(\F)\leq M(\F')$. Sometimes one can find such a $\F'$ which is significantly easier to analyse, but which still has $M(\F')$ fairly close to $M(\F)$. For instance, see remark \ref{remark:bigandsmallsets}.
    \item In a recent breakthrough paper, Frankston, Kahn, Narayanan and Park~\cite{fractional-exp-thresh} gives an upper bound on $M(\F)$ (and the corresponding threshold problem) in terms of the so-called \emph{spread} of $\F$. A family $\F$ is said to be $\kappa$-spread if no $r$-set $G\subset S$ occur as a subset in a more than a fraction $\kappa^{-r}$ of the members of $\FC$.\footnote{This is similar to the intuition that a function should not depend `too much' on any small set of coordinates. It would be interesting to see if there are any connections between the spread and the patchability of a family $\F$.}
    \item If one has shown that $\F$ is $(r,\smallcost,\eps)$-patchable for some not too large $r$ and $\smallcost$, then it suffices to upper bound $M(\Br)$ (the minimal cost of a $G$ within Hamming distance $r$ of $\FC$). In the case of $H$-factors which we discuss in section \ref{min-H-factor}, $M(\Br)$ was essentially already known for $r=\delta n$ for any fixed $\delta >0$.
\end{enumerate}

\subsection{Lower bound on $M(\F)$}
\label{section:lowerbound}

In this section we provide a general lemma that gives a lower bound on $M(\F)$ given some bound on the size of $\F$ and the sizes of members of $\F$. For brevity, we do this only for weights $X_i$ with distribution given by the $1/q$:th power of a $U(0,1)$-random variable. \todo{For other distributions satisfying $A(q)$?}

\begin{lemma}
\label{LB:general-set-family}
Let $\F$ be a family of subsets of $S$ such that each set $F\in\F$ has $\ell_0\leq |F|\leq \ell_1$ elements for some $\ell_0,\ell_1 >0$.
Assume $|\{ F\in \F:|F|=m\}| \leq c^m m^{\beta m }$ for some constants $c,\beta > 0$ and all $m$.

Then for any $t>0$ there exists a constant $c'>0$ such that $M(\F)$ (with respect to i.i.d.\ weigths $X_i$ satisfying $X_i^q\sim U(0,1)$) is at least $c'\cdot \min(\ell_0^{1-\beta/q},\ell_1^{1-\beta/q})$ with probability at least $1-\exp({- \ell_0 t})$.
\end{lemma}

For a set with $m$ elements, the probability that their total cost is below some given value $\bigcost\ll N:=|S|$ decays superexponentially fast as a function of $m$ (see claim \ref{claim:uniforms-sum}). On the other hand, the number of sets in $\F$ with $m$ elements, might increase superexponentially fast. If $\beta< q$, then the former decay rate beats the latter growth rate, so that sets with few elements dominate the expected number of `cheap' sets. In this case, $\ell_0^{1-\beta/q}\ll \ell_1^{1-\beta/q}$, so that $M(\F)=\Omp(\ell_0^{1-\beta/q})$. 
If instead $\beta\geq  q$, the large sets dominate, and $M(\F)=\Omp(\ell_1^{1-\beta/q})$.
 
\begin{proof}
We will apply the first moment method to the number of `cheap' sets in $\F$. 
For some $\bigcost$ to be determined later, let $R_m$ be the (random) number of $F\in \F$ with precisely $m$ elements and with $X_F\leq \bigcost$. (For $m< \ell_0$ or $m> \ell_1$, $R_m=0$.)

Then $R:= \sum_m R_m$ is the total number of sufficiently cheap sets, and by Markov's inequality $\Pr(M(\F)\leq \bigcost) \leq \E R $. 
Now, $\E [R_m] \leq {c^m m^{\beta m}\cdot \Pr(X_F\leq \bigcost)}$, where $F$ is any set with $m$ elements.
\begin{claim}
\label{claim:uniforms-sum}
If $U_i$ are i.i.d. uniform $[0,1]$ random variables and $q>0$ is a constant, then $\Pr(\sum_{i=1}^m U_i^{1/q}\leq \bigcost)\leq  \frac{\Gamma(1+q)^m}{\Gamma(1+qm)}\bigcost^m=2^{O(m)}\left(\bigcost/m\right)^{qm}$. 
\end{claim}
\begin{proof}[Proof of claim]
Let $A:=\{x\in \R_+^m: \sum_{i=1}^m |x_i|^{1/q}\leq \bigcost \}$ be the positive orthant of the $(1/q)$-norm ball with radius $\bigcost$, and let $B:=[0,1]^m$ be the unit box. The probability of the event $\{\sum_{i=1}^m U_i^{1/q}\leq \bigcost\}$ equals the $m$-dimensional volume $\mu(A\cap B)\leq \mu(A)$. 
The volume of $A$ is known to be $\mu(A)=\frac{\Gamma(1+q)^m}{\Gamma(1+qm)}\bigcost^{qm}$ (see \cite{vol-of-p-balls} for a concise derivation).
The asymptotic expression comes from noting that the numerator is $2^{O(m)}$, and using Stirling's approximation on the denominator.
\end{proof}

From the claim we have that $\Pr(X_F\leq \bigcost) \leq 2^{O(m)}(\bigcost/m)^{qm}$ (for $F$ with $|F|=m$) and hence $\E[R_m] \leq (c_0 \bigcost/m^{1-\beta/q})^{qm}$ for some $c_0$.
 If we pick $\bigcost:=c_1 c_0^{-1}\cdot \min(\ell_0^{1-\beta/q},\ell_1^{1-\beta/q})$ for some $c_1>0$, we get that $c_0\bigcost/ m^{1-\beta/q}\leq c_1$ for all $m$, and hence $\E [R_m] \leq c_1^{qm}$.
Thus $\E [X]=\sum_{m=\ell_0}^{\ell_1} \E[R_m] < c_1^{q\ell_0}/(1- c_1^q)$, which is less than $\exp(-\smallcost \ell_0)$ if $c_1$ is sufficiently small. 
\end{proof}

\section{Upper tail bound}
To apply Theorem \ref{thm:concentration} it is sometimes helpful to first have rougher bounds on the tails of $M(\F)$. Here we provide a simple bound on the upper tail.
\label{section:uppertail}
\begin{proposition}
If $\mu$ is the median of $M(\F)$, then for any $t\geq 0$,
\begin{equation}
    \Pr(M(\F)>t\mu )\leq 2^{1-t^q}.
    \label{bound:uppertail}
\end{equation}
Furthermore, $\E M(\F) \leq C_q\mu$ for some constant $C_q$ only depending on $q$, and thus \cref{bound:uppertail} holds with $\mu$ replaced with $\E M(\F)$ and the $2$ in the right-hand side replaced with some $r>1$.
\end{proposition}
\begin{proof}
Let $k:=\lfloor t^q\rfloor$, and assume WLOG that $k\geq 2$: If $t^q\leq 1$, then $2^{1-t^q}\geq 1$ and there is nothing to prove, and if $1<t^q<2$, then $t\mu>\mu$ and hence $\Pr(M(\F)>t\mu )\leq 1/2<2^{1-t^q}$. For each $i\in S$, and $j\in [k]$, let $X_i^{(j)}$ be i.i.d random variables following the same distribution as $X_i$.
As noted in remark \ref{remark:Aq-iter}, we can couple these random variables $X_i^{(j)}$ to $X_i$ such that surely 
$X_i\leq k^{1/q}\cdot \min_{j\in[k]}(X_i^{(j)}).$

Let $M^{(j)}$ be defined as $M(\F)$ but with edge weights $X_i^{(j)}$. Now, if $M^{(j)}= x$ for some $j,x$, by definition there exists an $F\in \F$ with $X_F^{(j)}= x$, and this $F$ then has weight $X_F=\sum_{i\in F}X_i\leq \sum_{i\in F}k\cdot X_i^{(j)}= kx$. This implies that surely
\[
M(\F)\leq X_F\leq k^{1/q}\cdot \min_{j\in[k]}(M^{(j)}),
\] 
and $k^{1/q}\leq t$ by the choice of $k$.
Noting that $\Pr(\min_{j\in[k]}(M^{(j)})> \mu) \leq 2^{-k}\leq 2^{1-t^q}$, the first part of the claim follows. 

For the `furthermore' part, we will use $\E M(\F)=\int_0^\infty \Pr(M(\F)> x)dx$. From the first part the integrand is at most $2^{1-(x/\mu)^q}$. After a change of variables $z=\ln(2)(x/\mu)^q$, its integral becomes $\int_0^\infty 2^{1-(x/\mu)^q}dx= {2}{\ln(2)^{-1/q}}\Gamma(1+1/q)\cdot\mu$. 
\end{proof}

\bibliographystyle{plain}
\bibliography{bib}
\end{document}